\documentclass[reqno,  twoside, a4paper, 11pt]{amsart}
\usepackage[english]{babel}
\usepackage{amsmath}
\usepackage[colorlinks=true,draft=false]{hyperref}
\usepackage{amsfonts}
\usepackage{thmtools}
\usepackage{amssymb}
\usepackage[pagewise,mathlines]{lineno}
\usepackage{tikz-cd}
\usepackage{amsthm}
\usepackage{fixme}
\usepackage{stmaryrd}
\usepackage{microtype}
\usepackage{mathrsfs}
\usepackage[all]{xy}
\usepackage{rotating}
\usepackage{enumitem}

\usepackage[nameinlink]{cleveref}
\newcount\hh
\newcount\mm
\mm=\time
\hh=\time
\divide\hh by 60
\divide\mm by 60
\multiply\mm by 60
\mm=-\mm
\advance\mm by \time
\def\hhmm{\number\hh:\ifnum\mm<10{}0\fi\number\mm}

\DeclareMathOperator{\trdeg}{trdeg}

\DeclareMathOperator{\tame}{tame}
\DeclareMathOperator{\nilp}{nilp}

\DeclareMathOperator{\codim}{codim}

\DeclareMathOperator{\Char}{char}

\DeclareMathOperator{\Gal}{Gal}

\DeclareMathOperator{\GL}{GL}

\DeclareMathOperator{\red}{red}

\DeclareMathOperator{\et}{\text{\'et}}

\newcommand{\Z}{\mathbb{Z}}

\newcommand{\N}{\mathbb{N}}
\newcommand{\C}{\mathbb{C}}

\renewcommand{\subset}{\subseteq}
\renewcommand{\supset}{\supseteq}

\renewcommand{\phi}{\varphi}

\declaretheoremstyle[spaceabove=15pt,spacebelow=15pt, qed=$\square$]{defstyle}
\declaretheoremstyle[
	spaceabove=15pt,
	spacebelow=15pt, 
	bodyfont=\itshape,
	qed=$\square$]{thmstyle}
\declaretheorem[style=thmstyle]{Theorem}
\declaretheorem[sibling=Theorem, style=thmstyle]{Proposition}
\declaretheorem[sibling=Theorem,  style=thmstyle]{Lemma}

\declaretheorem[sibling=Theorem, style=defstyle]{Definition}
\declaretheorem[sibling=Theorem, style=defstyle]{Remark}
\declaretheorem[sibling=Theorem, style=defstyle]{Example}
\setenumerate{label=(\alph*), ref=(\alph*)}

\title{Wild ramification of nilpotent coverings and coverings of bounded degree}
\address{Lars Kindler\\Freie Universit\"at Berlin\\ Mathematisches Institut\\ Arnimallee 3\\ 14195
Berlin, Germany}
\email{kindler@math.fu-berlin.de}
\urladdr{http://mi.fu-berlin.de/\textasciitilde kindler/}
\author{Lars Kindler}

\thanks{The author gratefully acknowledges support of the DFG (Deutsche Forschungsgemeinschaft) and Harvard University.}
\date{\today{ }\hhmm}

\begin{document}
\begin{abstract}
	A finite \'etale map between irreducible, normal varieties is called tame, if it is tamely ramified with respect to all partial compactifications whose boundary is the support of a strict normal crossings divisor. We prove that if the Galois group of a Galois covering contains a normal nilpotent subgroup of index bounded by a constant $N$, then the covering is tame if and only if it is tamely ramified with respect to a single distinguished partial compactification only depending on $N$. The main tools used in the proof are Temkin's local purely inseparable uniformization and a Lefschetz type theorem due to Drinfeld.
\end{abstract}
\maketitle
\section{Introduction}
Let $k$ be a field and let $X$ be a separated, finite type $k$-scheme.
Following \cite{Kerz/tameness}, a finite \'etale covering $f:Y\rightarrow X$ is
called \emph{tame} if for all $k$-morphisms $\phi:C\rightarrow X$, with $C$ a
regular $k$-curve, the induced covering $\phi_C:Y\times_X C\rightarrow C$ is
tamely ramified with respect to 
$\overline{C}\setminus C$, where $\overline{C}$ is the unique regular proper
curve with function field $k(C)$.  Thus to check whether a given covering
$Y\rightarrow X$ is tame, a priori, one has to test tameness along infinitely
many maps of curves to $X$.

%Here, we say that $(X,\overline{X})$ is a \emph{good partial compactification of $X$} (resp.~\emph{a good compactification}) if $\overline{X}$ is a regular (resp.~proper) $k$-scheme containing $X$ as an open dense
%subscheme, such that $\overline{X}\setminus X$ is the support of a strict normal
%crossings divisor.

If $X$ is regular and connected, it is known (\cite[Thm.~4.4]{Kerz/tameness}) that $f:Y\rightarrow X$ is
tame if and only if  $f$ is tamely ramified with respect to every geometric discrete valuation on $k(X)$. Equivalently, $f$ is tame if and only if it is tamely ramified with respect to every {good partial compactification} $(X,\overline{X})$ (\Cref{defn:gpc}), i.e., if the
normalization $\bar{f}:\overline{Y}\rightarrow \overline{X}$ of $\overline{X}$
in $Y$ is tamely ramified along the normal crossings divisor
$(\overline{X}\setminus X)_{\red}$. Again, to check whether a given covering is tame or not, one a priori has to check tameness with respect to infinitely many good partial compactifications.

However, if there exists a good compactification $(X,\overline{X}_0)$, that  is, a good partial compactification with  $\overline{X}_0$ \emph{proper},  then $f$ is tame if and only if $f$ is tamely ramified with respect to  this single good compactification.
Of course, as resolution of singularities is not known to be possible in positive
characteristic, it is not known whether every regular $k$-variety $X$ admits a
good compactification.

The goal of this short note is to prove that for a certain class of coverings, tameness can be detected by a single, privileged good \emph{partial} compactification.
%The goal of this short note is to prove that for a certain class of coverings of a smooth, quasi-projective variety over a finite field or over an algebraically closed field, there exists a single good \emph{partial} compactification, such
%that tameness of such coverings of $X$ can be detected by checking tameness with respect to
%this one privileged good partial compactification. 
\begin{Theorem}\label{thm:main}
	Assume that the field $k$ is either algebraically closed or finite.
	Let $X$ be a smooth, geometrically irreducible, quasi-projective $k$-scheme and
	$N\in \N$. There exists a good partial compactification
	$(X,\overline{X}_N)$ (\Cref{defn:gpc}), such that for any Galois \'etale covering $f:Y\rightarrow X$ with group $G$, such that $G$ has a normal \emph{nilpotent} subgroup of index $\leq N$,  the following are equivalent:
	\begin{enumerate}[label=\emph{(\alph*)}]
		\item $f$ is tame (\Cref{defn:tameness}, \ref{item:absoluteTameness}),
		\item $f$ is tamely ramified with respect to $(X,\overline{X}_N)$.
	\end{enumerate}
\end{Theorem}
The proof of the theorem uses
Temkin's inseparable local uniformization (\cite{Temkin}), following an idea of
Kerz-Schmidt, together with a generalization of a Lefschetz theorem due to Drinfeld (\cite[Prop.~6.2]{EsnaultKindler/Lefschetz}).\\

\begin{Example}
One particularly interesting class of Galois coverings to which \Cref{thm:main} applies arises as follows. Fix $r\in \N$. According to Jordan's theorem (\cite[Thm.~0.1]{LarsenPink}) there exists a constant $N(r)$, only depending on $r$, such that any finite subgroup of $\GL_r(\C)$ contains a normal abelian subgroup of index $\leq N(r)$. Thus, the tameness of representations $\rho:\pi_1(X,\bar{x})\rightarrow \GL_r(\C)$ with finite image can be detected on the good partial compactification $(X,\overline{X}_{N(r)})$.
\end{Example}

The structure of the article is as follows. In \Cref{sec:valuations} we gather some facts about discrete valuations and good partial compactifications, and in \Cref{sec:tameness} we present the necessary background from ramification theory. The proof of \Cref{thm:main} for nilpotent coverings is carried out in \Cref{sec:nilpotent}, and the general case is completed in \Cref{sec:proof}.

%Our goal is to prove that
%there are finitely many geometric discrete rank $1$ valuations $v_1,\ldots,
%v_n$ on $k(X)$, such
%that a finite Galois \'etale morphism $f:Y\rightarrow X$ with \emph{nilpotent
%Galois group} is tame if and only if the
%extension $k(X)\subset k(Y)$ is tamely ramified with respect to $v_1,\ldots,
%v_n$ (\Cref{thm:finiteness}). The argument follows an idea due to Moritz Kerz and Alexander Schmidt. 
%
\section{Geometric discrete valuations}
\label{sec:valuations}

In this section, $k$ is an arbitrary field and $X$ denotes a normal, irreducible,
separated $k$-scheme of finite type.
\begin{Definition} Let $k(X)$ denote the function field of $X$.
	\begin{enumerate}
		\item By a \emph{discrete valuation} on $k(X)$, we mean a discrete valuation of rank $1$, i.e.~with
			value group isomorphic to $\Z$.
		\item Let $v$ be a discrete valuation on $k(X)$. We write $\mathcal{O}_v\subset
			k(X)$ for its valuation ring, $\mathfrak{m}_v$ for the
			maximal ideal of $\mathcal{O}_v$ and $k(v)$ for its residue field.
		\item A \emph{model of $k(X)$} is separated, integral, finite type
			$k$-scheme $X'$ with $k(X')=k(X)$.
		\item If $X'$ is a model of $k(X)$ and $v$ a discrete valuation on $k(X)$, then a point $x\in
			X'$ is called \emph{center of $v$}, if
			$\mathcal{O}_{X',x}\subset \mathcal{O}_v\subset k(X)$,
			and $\mathfrak{m}_v\cap
			\mathcal{O}_{X',x}=\mathfrak{m}_x$. 
		\item $v$ is called \emph{geometric} or \emph{divisorial} or \emph{of the first
			kind} if there
			exists a model $X'$ of $k(X)$ such that $v$ has a
			center  which is a codimension $1$ point.	
	\end{enumerate}
\end{Definition}
\begin{Remark}Recall that if
			$X'$ is separated over $k$, then $v$ has at most one
			center on $X'$, and if $X'$ is proper, then $v$ has
			precisely one center on $X'$.
	\end{Remark}

The question of whether a discrete valuation is geometric or not is encoded in
the \emph{Abhyankar inequality}:

\begin{Proposition}[{\cite[Ch.~8,
	Thm.~3.26]{Liu/2002}}]\label{prop:abhyankar-inequality}
	If $v$ is a discrete valuation on $k(X)$ with center $x$ on $X$, then we have the inequality
	\begin{equation}\label{abhyankar-inequality}\trdeg_{k(x)}k(v)\leq \dim
		\mathcal{O}_{X,x} -1.
	\end{equation}
	The discrete valuation $v$ is geometric if and only if equality holds
	in \eqref{abhyankar-inequality}.
\end{Proposition}

From this one derives the following fact.
\begin{Proposition}[{\cite[Ch.~8,
	Ex.~3.16]{Liu/2002}}]\label{prop:dominant-geometric}
	If $f:Y\rightarrow X$ is a dominant morphism of normal, irreducible, separated,
	finite type $k$-schemes, and if $v$ is a geometric discrete valuation
	of $k(Y)$, then either $\mathcal{O}_v\supset k(X)$, or $w:=v|_{k(X)}$ is a geometric discrete valuation of
	$k(X)$ with finitely generated residue extension $k(w)/k(x)$.
\end{Proposition}

One can translate from the language of geometric discrete valuations into the
language of good partial compactifications.

\begin{Definition}\label{defn:gpc}
	Let $k$ be an arbitrary field  and let $X$ be a regular, irreducible, separated $k$-scheme of finite type.
Consider  a regular, separated $k$-scheme $\overline{X}$ of finite type,
			together with a dominant open immersion
			$X\hookrightarrow \overline{X}$, such that the complement
			$\overline{X}\setminus X$ is the support of a strict normal crossings
			divisor (\cite[\S 1.8]{Grothendieck/Tame}). The pair $(X,\overline{X})$ is said to be a
			\emph{good partial compactification of $X$}. If
			$\overline{X}$ is proper, then
			$(X,\overline{X})$ is said to be a \emph{good compactification} of $X$.
%		\item Write $\GPC(X)$ for the set of good partial
%			compactifications of $X$.
%		\item If  $(X,\overline{X}_1),
%			(X,\overline{X}_2)\in \GPC(X)$, then we write
%			$(X,\overline{X}_1)\leq (X,\overline{X}_2)$ if every
%			geometric discrete valuation which has a codimension $1$ center on
%			$\overline{X}_1$ also has a codimension $1$ center on $\overline{X}_2$.
%		\item If $(X,\overline{X}_1)\leq (X,\overline{X}_2)$ and
%			$(X,\overline{X}_2)\leq (X,\overline{X}_1)$, then we
%			write $(X,\overline{X}_1)\sim (X,\overline{X}_2)$ and
%			say that the two good partial compactifications of $X$
%			are equivalent.  
%	\end{enumerate}
\end{Definition}
Given a proper normal model $X'_0$ of $k(X)$ and a geometric discrete valuation $v$ on $k(X)$, for $n>0$  let $X'_n$ be the blow-up of $X'_{n-1}$ in the center of $v$. According \cite[Ch.~8, Ex.~3.14]{Liu/2002},  this process terminates and one obtains a model $X''$ on which $v$ has a codimension $1$ center. This implies the following proposition.
\begin{Proposition}\label{prop:blowup}
	Let $X$ and $k$ be as in \Cref{defn:gpc}. 
	%\begin{enumerate}[label=\emph{(\alph*)}]
		%\item Equivalence of good partial compactifications is an equivalence relation on $\GPC(X)$ and $(\GPC(X)/\!\!\sim, \leq)$ is a directed set.
		If $v=\{v_1\ldots, v_n\}$ is a set of
			geometric discrete valuations on $k(X)$ which do not have a center on $X$, then
			there exists a good partial compactification
			$(X,\overline{X}_{{v}})$ of $X$, such that the codimension $1$
			points on $\overline{X}_{{v}}\setminus X$ precisely correspond to
			$\{v_1,\ldots, v_n\}$.
			%, and every good partial
%			compactification $(X,\overline{X}')$ with this property
%			is equivalent to $(X,\overline{X}_v)$.
%			
%			This construction produces an isomorphism of 
%			directed sets \[(\GPC(X)/\!\!\sim,\leq)\cong
%			\left(\left\{\begin{tabular}{c}finite sets of\\  geometric discrete
%					valuations on $k(X)$\\ without center
%					on $X$\end{tabular}\right\},\subset\right).\]
%	\end{enumerate}
\end{Proposition}

\section{Tame ramification}
\label{sec:tameness}
We continue to denote by $k$ a field.
\begin{Definition}\label{defn:tameness}
	Let $X$ and  $Y$ be  normal, irreducible, separated, finite type $k$-schemes
	and $f:Y\rightarrow X$ a finite \'etale $k$-morphism.
	\begin{enumerate}
		\item If $v$ is a discrete valuation on $k(X)$, there exist
			finitely many discrete valuations $w_1,\ldots, w_r$ on
			$k(Y)$ such that $\mathcal{O}_{w_i}\cap
			k(X)=\mathcal{O}_v$. Then $f$ is
			\emph{tamely ramified with respect to $v$} if the
			residue extensions $k(w_i)/k(v)$ are separable and if
			the ramification indices on $\mathcal{O}_v\subset
			\mathcal{O}_{w_i}$ are prime to $\Char(k)$.

			Note that if $v$ has center on $X$, then $f$ is tame
			with respect to $v$, as $f$ is \'etale.
		\item If $(X,\overline{X})$ is a good partial compactification
			of $X$ (\Cref{defn:gpc}), then $f$ is said to be \emph{tamely ramified with respect to
			$(X,\overline{X})$} if $f$ is tamely ramified with
			respect to all discrete valuations on $k(X)$ which are
			centered in codimension $1$ points of 
			$\overline{X}$. Or, equivalently, if $f$ is tamely ramified with respect to the strict normal crossings divisor $(\overline{X}\setminus X)_{\red}$ (\cite[\S 2]{Grothendieck/Tame}). 

			If $\bar{x}$ is a geometric point of $X$, we write
			$\pi_1^{\tame}( (X,\overline{X}),\bar{x})$ for the
			quotient of $\pi^{\et}_1(X,\bar{x})$ corresponding to the
			category of finite \'etale coverings which are tamely
			ramified with respect to $(X,\overline{X})$. 
			Concretely, $\pi_1^{\tame}( (X,\overline{X}), \bar{x})$
			is the quotient of $\pi_1^{\et}(X,\bar{x})$ by the
			smallest closed normal subgroup containing all wild
			inertia groups attached to points of the normalization of $\overline{X}$ in $Y$ lying over the codimension $1$ points of
			$\overline{X}\setminus X$.
		\item\label{item:absoluteTameness} The morphism $f$ is called \emph{tame} if for every
			regular $k$-curve $C$ and for every $k$-morphism
			$\phi:C\rightarrow X$, the induced \'etale covering
			\[f_C:Y\times_X C \rightarrow C\] is tamely ramified with
			respect to the good compactification $(C,\overline{C})$,
			where $\overline{C}$ is the unique normal proper curve
			with function field $k(C)$.

			If $\bar{x}$ is a geometric point of $X$, we denote by
			$\pi_1^{\tame}(X,\bar{x})$ the quotient of
			$\pi_1(X,\bar{x})$ corresponding to the category of
			tame \'etale coverings of $X$ (\cite[Section 7]{Kerz/tameness}).
		\item As is customary, we will often use the word ``wild'' in place of 
			``not tame''.
	\end{enumerate}
\end{Definition}
\begin{Remark}\label{rem:kerzschmidt}
	In \cite[Rem.~4.5]{Kerz/tameness} it is proved that if an \'etale covering $f$ as above is
	tame, then $f$ is tamely ramified with respect to all geometric
	discrete valuations $v$ on $k(X)$, or equivalently with respect to all
	good partial compactifications of $X$. Moreover, if  $X$ is regular, then 
	\cite[Thm.~4.4]{Kerz/tameness} shows that the converse is also true.

	In particular, if $\bar{x}$ is a geometric point of $X$ and if $(X,\overline{X})$ is a good partial compactification of $X$, then there is a canonical quotient map $\pi_1^{\tame}( (X,\overline{X}),x)\twoheadrightarrow \pi_1^{\tame}(X,\bar{x})$, which is an isomorphism if $\overline{X}$ is proper.
\end{Remark}

In the proof of \Cref{thm:main} we will use the following two lemmas  for which we
could not find a reference.

\begin{Lemma}\label{lem:SES}Let $k$ be a perfect field and  fix an algebraic closure $\bar{k}$ of $k$. Let $X$ be a geometrically irreducible, regular, separated, finite type $k$-scheme. If $\bar{x}$ is a geometric point of $X\times_k \bar{k}$, then there is a short exact sequence
	\begin{equation*}
		\begin{tikzcd}
			1\rar&\pi_1^{\tame}(X\times_k \bar{k},\bar{x})\rar&\pi_1^{\tame}(X,\bar{x})\rar&\Gal(\bar{k}/k)\rar&1
		\end{tikzcd}
	\end{equation*}
\end{Lemma}
\begin{proof}
	The argument is completely analogous to \cite[XI.6.1]{SGA1}. 
	Inspection of this reference shows that the only thing to
	be checked is that if $f:Y\rightarrow X_{\bar{k}}$ is a tame Galois
	covering, then $f$ is the base change of a tame covering of $X_{k'}$
	for some finite extension $k'/k$. It is clear that $f$ comes via base
	change from a Galois covering $Z\rightarrow X_{k'}$ for some $k'$. Let
	$\phi:C\rightarrow X_{k'}$ be a nonconstant map from a regular
	$k$-curve $C$ to $X_{k'}$. By enlarging $k'$ we may assume that the
	points of $\overline{C}\setminus C$ are $k'$-rational, where
	$\overline{C}$ is the unique normal proper $k'$-curve compactifying
	$C$, and similarly for the points at infinity of $Z\times_{X_{k'}} C$ lying over
	$\overline{C}\setminus C$. 
	
	Since a finite, totally ramified map of discrete valuation rings is tamely
	ramified if and only if the corresponding map on strict henselizations is
	tamely ramified, it follows that $\phi_Z:C\times_{X_{k'}}Z\rightarrow C$ is
	tame if and only if its base change to $\bar{k}$ is tamely ramified.

	Consequently, $Z\rightarrow X_{k'}$ is tame, which is what we set out to prove.
\end{proof}
\begin{Lemma}\label{lem:tamevsPI}Let $K$ be a field equipped with a discrete valuation $v$, $L/K$ a
	finite separable extension and $K'/K$ a finite purely inseparable
	extension. Let $v'$ denote the unique extension of $v$ to $K'$. Then
	$v$ is tamely ramified in $L/K$ if and only if $v'$ is tamely ramified
	in $L\otimes_K K'/K'$.
\end{Lemma}
\begin{proof}First note that $L':=L\otimes_K K'$ really is a field, as $K'$ and
	$L$ are linearly disjoint when embedded in any algebraic closure of $K$. Without loss of generality we may assume
	that 
	$K'=K[X]/(X^p-a)$ with $a\in K\setminus K^p$.  Let $w$ be a valuation on $L$
	extending $v$ and $w'$ the unique extension of $w$ to $L'$. Note that
	the completion $K'_{v'}$ is either $K_v$ or $K_v[X]/(X^p-a)$,
	depending on whether $a$ is a $p$-th power in $K_v$ or not. As
	$L_w/K_v$ is separable, $a$ is a $p$-th power in $L_w$ if and only if
	it is a $p$-th power in $K_v$. Thus in the commutative diagram 
	\begin{equation*}
		\begin{tikzcd}
			L_w\ar{r}&L'_{w'}\\
			K_v\rar \uar&K'_{v'}\uar
		\end{tikzcd}
	\end{equation*}
	the horizontal arrows are either both purely inseparable of degree $p$
	or both isomorphisms. In either case 
	$L'_{w'}=L_w\otimes_{K_v}K'_{v'}$.  Write $e$ and $e'$ for the
	ramification indices of $w|v$ and $w'|v'$, and $f,f'$ for the residue
	degrees of these extensions. Then (\cite[II, Cor.~1, p.~29]{Serre/LocalFields})
	\[ef=[L_w:K_v]=[L'_{w'}:K'_{v'}]=e'f'.\]
	
	Consider the diagram of the
	residue extensions
	\begin{equation*}
		\begin{tikzcd}
			k(w)\ar{r}&k(w')\\
			k(v)\rar \uar&k(v')\uar.
		\end{tikzcd}
	\end{equation*}
	In general, it is not true that 
	$k(w')=k(w)\otimes_{k(v)}k(v')$.

	We know that the bottom horizontal arrow is either an isomorphism or purely inseparable of degree $p$.
	If $w|v$ is tamely ramified, then $k(w)/k(v)$ is separable, so again
	we see that the horizontal arrows are either isomorphisms or both
	purely inseparable of
	degree $p$ and $k(w')=k(w)\otimes_{k(v)}k(v')$. In either case $f=f'$ and
	$k(w')/k(v')$ is separable. From $ef=e'f'$ it follows that $(e',p)=1$
	and hence that $w'|v'$ is tamely ramified.

	Conversely, assume that $w'|v'$ is tamely ramified. Let $k^s\subset
	k(w)$ be the separable closure of $k(v)$ in $k(w)$. By multiplicativity of the
	separable degree, we see that $f'=[k^s:k(v)]$. Thus
	$f=f'p^{\varepsilon}$ where $\varepsilon=0$ or $1$. But by assumption
	$(e',p)=1$, so from $ef=e'f'$ it follows that $\varepsilon=0$, $f=f'$
	and $e'=e$. In particular, $w|v$ is tamely ramified.
	
	%%	
%	
%	Moreover, we may assume that $a$ is not
%	a $p$-th power in $K$, i.e.~that $[K':K]=[L':L]=p$.
%
%	Assume that $L/K$ is tamely ramified. Let $w$ be a valuation on $L$
%	extending $v$ and let $v'$, $w'$ be the unique extensions to $K'$,
%	$L'$, and $k(v), k(w), k(v'),k(w')$ the respective residue fields. As
%	$k(w)/k(v)$ is separable, the image $\bar{a}\in k(v)$ of $a$ is a
%	$p$-th root in $k(v)$ if and only if it is a $p$-th root in $k(w)$. It
%	follows that $k(w')=k(w)\otimes_{k(v)}k(v')$, so $k(w')/k(v')$ is
%	separable.
%	Let $e$ denote the ramification index of $L/K$, $f=[k(w):k(v)]$, and
%	$e',f'$ the analogous invariants of $L'/K'$.  We just saw that $f=f'$.
%	On the other hand, looking at the henselizations, we see that
%	$ef=e'f'$, so $e'$ is prime to $p$ and hence $L'/K'$ is tamely
%	ramified with respect to $w$ and $v$.
%
\end{proof}

\section{Nilpotent coverings}
\label{sec:nilpotent}
In this section we prove a special case of \Cref{thm:main}.
\begin{Proposition}\label{prop:nilpotent}
	Let $k$ be a field and $X$ a regular, irreducible, separated, finite
	type $k$-scheme. There exists a good partial compactification
	$(X,\overline{X}_{\nilp})$ such that for every finite \'etale covering
	$f:Y\rightarrow X$, which can be \emph{dominated by a nilpotent Galois
	covering}, the following are equivalent:
\begin{enumerate}[label=\emph{(\alph*)}]
		\item $f$ is tame.
		\item $f$ is tamely ramified with respect to
			$(X,\overline{X}_{\nilp})$.
	\end{enumerate}
	Moreover, if $X$ is quasi-projective, then we can choose $\overline{X}_{\nilp}$ to be quasi-projective as well.
\end{Proposition}
%\begin{Remark}
%	Of course we would like to get rid of the nilpotency condition; at
%	present, the author does not know how to do that without assuming 
%	resolution of singularities. Below
%	the proof of \Cref{thm:finiteness} we will show what goes wrong
%	in the given proof.
%\end{Remark}

The proof of this result follows an idea suggested by Moritz Kerz. Before we proceed
with the argument, we recall the following main ingredient.
\begin{Theorem}[{``Inseparable local uniformization'',
		\cite[Cor.~1.3.3]{Temkin}}]\label{thm:Temkin}
	Let $X$ be an integral,  separated finite type $k$-scheme.
	Then there exist morphisms $\phi_i:V_i\rightarrow X$, $i=1,\ldots, r$,
	with the following properties.
	\begin{enumerate}[label=\emph{(\alph*)},ref=(\alph*)]
		\item\label{temkin1} The $V_i$ are \emph{regular}, integral, separated, finite type
			$k$-schemes.
		\item\label{temkin2} The maps $V_i\rightarrow X$  are dominant and
			of finite type. They \emph{cover} $X$ in
			the following sense: Any valuation on $k(X)$ (not necessarily
			discrete or rank $1$) with center on $X$ lifts
		to a valuation on some $k(V_i)$ with center on $V_i$.
	\item\label{temkin3} The induced extensions $k(X)\subset k(V_i)$ are finite and
			\emph{purely inseparable}, $i=1,\ldots, r$.
	\end{enumerate}
\end{Theorem}

\begin{proof}[Proof of \Cref{prop:nilpotent}]
	Let $\overline{X}$ be a normal compactification of $X$ and apply
	\Cref{thm:Temkin}  to $\overline{X}$.  We obtain a map
	$\phi=\coprod_{i=1}^r\phi_i:\coprod_{i=1}^rV_i\rightarrow \overline{X}$, satisfying
	\ref{temkin1} -- \ref{temkin3}. Write
	$U_i:=\phi_i^{-1}(X)$. Then on each $V_i$ there are finitely many
	codimension $1$ points not contained in $U_i$. They correspond to
	geometric discrete valuations on $k(V_i)$, and hence give rise to a
	finite number of geometric discrete valuations $v_1,\ldots, v_n$ on $k(X)$
	(\Cref{prop:dominant-geometric}). It suffices to prove that a nilpotent Galois covering
	$f:Y\rightarrow X$ is  tame if and only if it is tamely ramified with respect to
	$v_1,\ldots, v_n$, or equivalently, with respect to any good partial
	compactification $(X,\overline{X}_{\nilp})$ on which $v_1,\ldots, v_n$
	have centers in codimension $1$ points.
	
	As $X$ is regular, \Cref{rem:kerzschmidt} shows that $f$ is tame if and only if it
is tamely ramified with respect to all geometric discrete valuations of $k(X)$, in particular with
respect to $v_1,\ldots, v_n$. 

Assume conversely that there is a geometric discrete valuation $v$ on $k(X)$ with
respect to which $f$ is wildly ramified. We want to show that $f$ is wildly
ramified with respect to one of the valuations $v_1,\ldots, v_n$. As
$\overline{X}$ is proper, $v$ has a center $x$ on $\overline{X}$ (possibly
with
$\codim_{\overline{X}} \overline{\{x\}}>1$). It follows
that there exists a valuation $w$, on, say, $k(V_1)$, extending $v$ and 
having a center on $V_1$. As
$k(V_1)/k(X)$ is purely inseparable, $w$ is the only extension of $v$ to
$k(V_1)$.

Write $f_1:{Y}_1\rightarrow U_1$ for the base change of
${f}$ to $U_1=\phi_1^{-1}(X)$. According to \Cref{lem:tamevsPI}, $f_1$ is  wildly 
ramified with respect to $w$ and we claim that this implies that $f_1$ is
wildly ramified with respect to a codimension $1$ point lying on $V_1\setminus
U_1$. Applying \Cref{lem:tamevsPI} again, this would imply that $f$ is wildly ramified with respect to one of the
valuations $v_1,\ldots, v_n$, which is what we want to prove.

To prove the claim, we proceed along the lines of \cite[Prop.~1.10]{Schmidt}.
Let $G=\Gal(Y/X)$. As $G$ is nilpotent, $G\cong P\times P'$ with $P$ a
$p$-group and $P'$ a group of order prime to $p:=\Char(k)$.
In particular, the covering
$Y_1\rightarrow U_1$ can be written as a tower of Galois coverings $f_1:Y_1\xrightarrow{a}
Y'_1\xrightarrow{b} U_1$, with $\Gal(b)=P$ and $\Gal(a)=P'$.
As the discrete valuation $w$ of $k(V_1)$ is wildly ramified in $k(Y_1)$, it
follows that $w$ is ramified in $k(Y'_1)$. On the other hand, as $w$ has a
center on the regular scheme $V_1$, the Zariski-Nagata purity theorem (\cite[X, Thm.~3.1]{SGA1}) implies that the normalization of $V_1$ in
$k(Y'_1)$ is ramified over a closed subscheme of pure codimension $1$
contained in $V_1\setminus U_1$. As
$Y'_1\rightarrow U_1$ has $p$-power degree, this ramification is wild. This proves the claim.

Finally, if $X$ is quasi-projective, then we can choose $\overline{X}$ to be a projective, normal compactification of $X$. The above construction applied to $\overline{X}$ yields a quasi-projective $\overline{X}_{\nilp}$ with the desired properties. 

%
%Let
%$w'$ be a discrete valuation of $Y_1:=V_1\times_X Y$ lying over $w$ such that
%$w'/w$ is wildly ramified. 
%Let
%$P\subset G$ be a cyclic subgroup of order $p=\Char(k)$, contained in the wild inertia group subgroup of
%$G$ attached to $w'$. Then $Y_1\rightarrow Y_1/H$ is Galois \'etale of degree $p$, and  ramified in a point
%of $Y_1/H$ lying over the center of $w$ on $V_1$. Indeed, as $w$
%has a center on $V_1$, $w'$ has a center on $Y_1$, and thus $w'|_{Y_1/H}$ has
%a center $y$ on $Y_1|_H$.  It follows that $Y_1\rightarrow  Y_1/H$ is ramified
%over $y$. 
%
%As
%$Y_1/H$ is regular, Zariski's purity theorem implies that there is
%a codimension $1$ point $\eta$ on $(V_1\times_X Y)/H$ over which $V_1\times_X Y$ is
%ramified. In fact the map is \emph{wildly ramified} over $\eta$ as $H$ has
%order $p$. This
%means that $f_Y$ is wildly ramified over a codimension $1$ point of $V_1$
%(namely the image of $\eta$). The restriction to $\overline{X}$ of the
%valuation associated with $\eta$ is one of the valuations $v_1,\ldots, v_n$,
%so a final application of \autoref{lem:tamevsPI} implies that $f$ is wildly
%ramified with respect to one of the $v_1,\ldots, v_n$, which is what we wanted
%to prove.
\end{proof}

%\begin{Remark}
%	We explain the difficulty in generalizing this argument to arbitrary
%	Galois coverings.
%	Continue to use the notations of the proof but assume that
%	$f:Y\rightarrow X$ is Galois with arbitrary Galois group $G$. Again
%	let $v$ be a geometric discrete valuation of $k(X)$ which
%	ramifies wildly in $f$ and let $w$ be the corresponding wildly
%	ramified valuation on, say, $k(V_1)$. One could then try to argue as
%	follows: Let $w'$ be a valuation on $k(Y_1)$ extending $w$, such that
%	$w'|w$ is wildly ramified. Let $P$ be the wild inertia subgroup
%	of $G$ associated with $w'|w$. Then $f_1$ factors $Y_1\xrightarrow{a}
%	Y'_1\xrightarrow{b} U_1$, now with $\Gal(a)=P$, and $b$ not
%	necessarily Galois. The normalization $\overline{Y'}_1$ of $V_1$ in $k(Y'_1)$ is
%	generally not regular, so Zariski-Nagata purity does not apply. We
%	could use \Cref{thm:Temkin} to replace $\overline{Y'}_1$ by a regular scheme, but then
%	we would perhaps have to add more valuations to $v_1,\ldots, v_n$.
%\end{Remark}

\section{Applying a theorem of Drinfeld}
\label{sec:proof}
We begin by establishing the following consequence of a Lefschetz theorem of
Drinfeld.
\begin{Proposition}\label{prop:topFG}
	Assume that the field $k$ is either algebraically closed or finite. Let 
	$X'$ be a normal, projective, geometrically irreducible $k$-scheme and let $X\subset X'$ be a
	smooth, dense open subscheme together with a geometric point $\bar{x}\rightarrow
	X$. Let $\Sigma\subset X'$ be a
	closed subset satisfying the following conditions.
	\begin{enumerate}[label=\emph{(\alph*)}]
		\item $\codim_{X'} \Sigma \geq 2$, 
		\item $X'\setminus \Sigma$ is smooth,
		\item $(X'\setminus X)\setminus \Sigma$ is the support of a smooth divisor. 
	\end{enumerate}
	Then $(X, X'\setminus \Sigma)$ is a good partial compactification and
	the 
	profinite group $\pi^{\tame}_1((X,X'\setminus \Sigma),\bar{x})$
(\Cref{defn:tameness}) is topologically finitely generated.

	In particular, for a given $N\in \N$, the group  $\pi^{\tame}_1( (X,X'\setminus \Sigma),\bar{x})$ only has finitely many open subgroups of index $\leq N$.
\end{Proposition}
\begin{Remark}
A profinite group which has only finitely many quotients of a given
	cardinality is called \emph{small}. For a related, but smallness
	result for fundamental groups of varieties over finite fields, see
	\cite{Hiranouchi}, which utilizes a deep finiteness theorem of Deligne (\cite{Esnault/Kerz}).
\end{Remark}
\begin{proof}
	By \cite[Prop.~6.2]{EsnaultKindler/Lefschetz}, which is a generalization of  \cite[Appendix C]{Drinfeld}, there exists a closed, smooth, irreducible $k$-curve $C\subset X'\setminus \Sigma$ not contained in $X'\setminus X$ and intersecting $((X'\setminus \Sigma)\setminus X)_{\red}$ transversely,  such that for every finite irreducible \'etale covering $Y\rightarrow X$ which is tamely ramified with respect to $(X,X'\setminus \Sigma)$, the pullback $Y\times_X (C\cap X)\rightarrow (C\cap X)$ is irreducible. This means that for a geometric point  $\bar{c}$ of $C\cap X$, the induced map
	\[\pi_1^{\tame}(C\setminus X,\bar{c})\rightarrow \pi_1^{\tame}( (X,X'\setminus
		\Sigma),\bar{c})\]
	is surjective. The group on the left is known to be topologically
	finitely generated if $k$  is algebraically closed or finite (\cite[XIII, Thm.~2.12]{SGA1} and \Cref{lem:SES}).

	Finally, according to \cite[Lemma 16.10.2]{Jarden/FieldArithmetic}, a topologically finitely generated group only
	has finitely many open subgroups of index $\leq N$.
\end{proof}

\begin{proof}[Proof of \Cref{thm:main}]
	Fix a geometric point $\bar{x}$ of $X$ and
	let $X'$ be any projective normal compactification of $X$. We ``approximate'' the good partial compactification $X_N$ in several steps. First, let $X'_{N}$ be an open subset of $X'$ containing $X$, such that $(X,X'_N)$ is a good partial compactification and such that $\codim_{X'}(X'_N\setminus X)\geq 2$. According to \Cref{prop:topFG}, there are only finitely many open normal subgroups of index $\leq N$ in  $\pi_1^{\tame}( (X,X'_N),\bar{x})$.
In particular, there are only finitely many Galois coverings $Y_1/X,\ldots, Y_m/X$, $m\in \N$, of degree $\leq N$, which are tamely ramified with respect to $(X,X'_{N})$ but not tame in the sense of \Cref{defn:tameness}, \ref{item:absoluteTameness}. According to \Cref{rem:kerzschmidt}, for every $i\in \{1,\ldots, m\}$ we find a geometric discrete valuation $v_i$ on $k(X)$ with respect to which $Y_i/X$ is wildly ramified. Blowing up $X'$ repeatedly in the centers of the $v_i$ and normalizing, we obtain a normal projective compactification $X''$ of $X$, such that each $v_i$ is centered in a codimension $1$ point of $X''$ (see \Cref{prop:blowup}). Let   $X''_N$  be an open subset of $X''$ containing $X$ and all codimension $1$ points of $X''$, such that $(X,X''_N)$ is a good partial compactification of $X$. It follows that a Galois covering of degree $\leq N$ of $X$ is tame if and only if it is tamely ramified with respect to the good partial compactification $(X,X''_N)$.

Now let $f:Y\rightarrow X$ be a Galois \'etale covering with group $G$ such that $G$ contains a normal nilpotent subgroup $H$ of index $\leq N$. It factors as 
\begin{equation}\label{eq:factorization}
		\begin{tikzcd}
		Y\ar[swap]{dd}{f}\ar{dr}{f_H}\\
		&Y'\ar{dl}{f_{G/H}}\\
		X
	\end{tikzcd}
\end{equation}
where $f_H$ is Galois \'etale with group $H$ and $f_{G/H}$ is Galois \'etale with group $G/H$; in particular, $\deg(f_{G/H})\leq N$. If $f$ is tamely ramified with respect to $(X,X''_N)$, then so is $f_{G/H}$, and using \Cref{prop:topFG} again, we see that there are only finitely many possibilities for $f_{G/H}$; we write them as $Y'_1/X,\ldots, Y'_{m'}/X$. %Note that by construction of $X''_N$, the coverings $Y_1,\ldots, Y_m$ are tame in the sense of \Cref{defn:tameness}, \ref{item:absoluteTameness}.

For $i=1,\ldots, m'$, we apply \Cref{prop:nilpotent} to the $k$-scheme $Y'_i$ and obtain a finite set $\{w_{ij}\}_{j=1,\ldots, r_i}$ of geometric discrete valuations on $k(Y'_i)$ with the property that a nilpotent covering of $Y'_i$ is tame if and only if it is tamely ramified with respect to $w_{ij}$, $j=1,\ldots, r_i$. As $Y'_i/X$ is finite,  \Cref{prop:dominant-geometric} shows that the restriction $v_{ij}:=(w_{ij})|_{k(X)}$ is a geometric discrete valuation on $k(X)$.

Now assume that $f$ is tamely ramified with respect to $(X,X''_N)$ and with respect to all $v_{ij}$.  As $\deg(f_{G/H})\leq N$, $f_{G/H}$ is tame by construction of $X''_N$, and $f_{G/H}$ is one of the coverings $Y'_1/X,\ldots, Y'_{m'}/X$. Moreover, as the valuations $w_{ij}$ arise from 
\Cref{prop:nilpotent}, it follows that if $f_{H}$ is wildly ramified, then there exists a pair $(i,j)$, such that $f_H$ is wildly ramified with respect to $w_{ij}$, which implies that $f$ is wildly ramified with respect to $v_{ij}$; contradiction. Thus $f$ is tame.

Finally, blow up $X''$ repeatedly in the centers of the $v_{ij}$ and normalize to obtain $X'''$, a normal projective compactification of $X$, such that each $v_{ij}$ is centered in a codimension $1$ point of $X'''$.   Let $X_N$ be a suitable open subset of $X'''$ containing $X$ and all the centers of the $v_{ij}$. We proved that $f$ as in \eqref{eq:factorization} is tame, if it is tamely ramified with respect to  $(X,X_N)$. The converse is true by definition.
\end{proof}
\providecommand{\bysame}{\leavevmode\hbox to3em{\hrulefill}\thinspace}
\providecommand{\MR}{\relax\ifhmode\unskip\space\fi MR }
% \MRhref is called by the amsart/book/proc definition of \MR.
\providecommand{\MRhref}[2]{%
  \href{http://www.ams.org/mathscinet-getitem?mr=#1}{#2}
}
\providecommand{\href}[2]{#2}

\end{document}